\documentclass[12pt]{article}
\usepackage{amssymb}
\usepackage{graphicx, enumerate}
\usepackage{amsfonts}
\usepackage{graphicx}
\usepackage{amsmath}
\usepackage{amsthm}
\usepackage{amssymb}
\usepackage[colorlinks]{hyperref}  
\usepackage{color}

    \newtheorem{theorem}{Theorem}
    \newtheorem{lemma}[theorem]{Lemma}
    \newtheorem{proposition}[theorem]{Proposition}

\theoremstyle{definition} 

\newcommand{\eps}{\varepsilon}

\newcommand{\lstar}{{\raise-0.15ex\hbox{$\scriptstyle \ast$}}}

\theoremstyle{remark} 

\newcommand{\ind}{{\bf 1}}

\newcommand{\sch}{\textup{Sch}}
\newcommand{\Sineb}{\textup{Sine}_\beta}

\newcommand{\WW}{\mathcal{W}}

\begin{document}

\title{Overcrowding asymptotics for the $\Sineb$ process}%

\author{Diane Holcomb\footnote{Department of Mathematics, University of Arizona,  dianeholcomb@math.arizona.edu} \and Benedek Valk\'o\footnote{Department of Mathematics, University of Wisconsin - Madison, valko@math.wisc.edu}}%




\maketitle

\begin{abstract}
We give overcrowding estimates for the $\Sineb$ process, the bulk point process limit of the Gaussian $\beta$-ensemble. We show  that the probability of having at least $n$ points in a fixed interval  is given by $e^{-\tfrac{\beta}{2} n^2 \log(n)+\mathcal O(n^2)}$ as $n\to \infty$. We also identify the next order term in the exponent if the size of the interval goes to zero.

\end{abstract}

\section{Introduction}

The Gaussian $\beta$-ensemble is  a family of point processes on the real line given by the joint density function
\begin{align}\label{beta}
p_{m,\beta}(\lambda_1, \dots, \lambda_m)=\frac{1}{Z_{m,\beta}} \prod_{1\le i<j\le m} |\lambda_i-\lambda_j|^\beta\times  \prod_{i=1}^m e^{-\frac{\beta}{4} \lambda_i^2}.
\end{align}
Here $\beta>0$ and $Z_{m,\beta}$ is an explicitly computable normalizing constant. For the specific values of $\beta=1, 2$ and 4 one obtains the joint eigenvalue density of the classical random matrix ensembles GOE, GUE and GSE. 

We  study the $\Sineb$ process, the bulk scaling limit of the Gaussian $\beta$-ensemble. Fix $\beta>0$, and consider a sequence of finite point processes $\Lambda_{m,\beta}$ with distribution given by the density (\ref{beta}). Then as $m\to \infty$ the scaled processes $2\sqrt{m}\, \Lambda_{m,\beta}$ converge in distribution to a translation invariant point process which we call $\Sineb$. The existence and description of the limiting process in the $\beta=1, 2$ and 4 cases were obtained by Dyson, Gaudin and Mehta  (see e.g.~the monographs \cite{AGZ}, \cite{ForBook},  \cite{mehta}), while in the general $\beta$ case this was done in \cite{BVBV}. We note that in our normalization the  $\Sineb$ process has particle  density $\frac{1}{2\pi}$. 

We study the problem of overcrowding for the  $\Sineb$ process. We describe the asymptotic probability of having at least $n$ points in a fixed interval for a large $n$. Since the process is translation invariant we may assume that the interval is of the form $[0,\lambda]$ with $\lambda>0$. We will use the notation $N_\beta(\lambda)$ for the counting function of the process, for $\lambda>0$ this gives  the number of points in $[0,\lambda]$. Our main result is as follows. 

\begin{theorem}\label{thm:overcrowding}
Fix $\lambda_0>0$. Then there is a constant $c>0$ depending only on $\beta$ and $\lambda_0$ so that for any $n\ge 1$ and $0<\lambda\le \lambda_0$ we have
\begin{align}
e^{-\frac{\beta}{2} n^2 \log\left(\tfrac{n}{\lambda}\right)-c n \log(n+1) \log\left(\tfrac{n}{\lambda}\right)-c n^2}\le  P(N_\beta(\lambda)\ge n) \le e^{-\frac{\beta}{2} n^2 \log\left(\tfrac{n}{\lambda}\right)+c n \log(n+1) \log\left(\tfrac{n}{\lambda}\right)+c n^2}.
\end{align}  
\end{theorem}
The theorem implies that for a fixed $\lambda>0$ we have $P(N_\beta(\lambda)\ge n)=e^{-\frac{\beta}{2} n^2 \log n+\mathcal{O}(n^2)}$ as $n\to \infty$. If we let $n\to \infty$ and $\lambda\to 0$ simultaneously then we get   
\[
P(N_\beta(\lambda)\ge n)=e^{-\frac{\beta}{2} n^2 \log (n/\lambda)+\mathcal{O}(n^2+n\log n \log (\lambda^{-1}))}.
\]

Dyson  used Coulomb gas arguments in  \cite{Dyson62}  to predict the asymptotic size of $P(N_\beta(\lambda)=n)$ for a fixed $n$ and large $\lambda$. His predictions were made precise in \cite{BTW} for  $\beta=1, 2$ and 4 using the fact that in these classical cases the probabilities in question have  Fredholm determinant representations. For general $\beta$ the special case of $n=0$ (i.e.~the large gap probability) was treated in \cite{BVBV2}.   

Fogler and Shklovskii in \cite{FS95} extended Dyson's method to give predictions on the asymptotic  size of  $P(N_\beta(\lambda)=n)$ in the regime where $n$ and $\lambda$ are large. Although this does not cover the case we are interested in ($\lambda$ bounded, $n$ large), our results agree with  their prediction in the case when $n \gg \lambda$. Section 14.6 of \cite{ForBook} gives a nice overview of the Coloumb gas method in the study of the distribution of $N_\beta(\lambda)$.

Our proof uses techniques that we developed in \cite{HV}, where we proved a large deviation principle for $\frac{1}{\lambda}N_\beta(\lambda)$ with scale $\lambda^2$ and a good rate function of the form $\beta I(\rho)$.   One of the consequences of this large deviation  result is that  if $\rho>\frac{1}{2\pi}$ then $P(N_\beta(\lambda)\ge \rho \lambda)$ decays asymptotically as $e^{-\beta \lambda^2 I(\rho)(1+o(1))}$ as $\lambda\to \infty$. It was  shown in \cite{HV} that $I(\rho)$ grows like $\frac{1}{2}\rho^2 \log\rho$ as $\rho\to \infty$ so the decay of $P(N_\beta(\lambda)\ge \rho \lambda)$  is formally consistent  with  our Theorem \ref{thm:overcrowding}, even though  it describes a different regime of $n,\lambda$. 

Rigorous overcrowding estimates have been obtained for certain non-Hermitian random matrix models. Here one is interested in the asymptotic probability of finding at least $n$ points in a  fixed disk. In \cite{K06} Krishnapour identified the overcrowding asymptotics for the complex Ginibre ensemble and also gave similar estimates for zeros of certain Gaussian analytic functions. In \cite{AS13} and \cite{AIS14} these results were generalized to point processes obtained by considering products of complex and quaternion Ginibre ensembles. It is interesting to note that in all of these non-Hermitian cases the probability of finding at least $n$ points in a fixed disk decays like $e^{-c n^2 \log n(1+o(1))}$ where $c$ is a parameter depending only on the model and not the size of the disk. This  is of the same form as   our result for the $\Sineb$ process. 

The finite version of our problem would be to study the probability of having a large number of points in a small interval of order $m^{-1/2}$ in the bulk of the support of (\ref{beta}). This problem has been considered for the spectrum of general Wigner matrices in the context of showing the universality of the bulk limit process, see e.g.~Theorem 3.5 of \cite{ESY} or Corollary B.2 of \cite{BEYY}. These results concentrate on giving sufficiently strong (but not necessarily optimal) upper bounds for the overcrowding probabilities. 

\bigskip

The starting point of our proof is the following characterization of the counting function of $\Sineb$, due to \cite{BVBV}. 
\begin{theorem}[\cite{BVBV}]\label{thm:sine}
Let $\lambda>0$ and consider the strong solution of the following diffusion:
\begin{align}
\label{sinesde}
d\alpha_\lambda = \lambda \frac{\beta}{4} e^{-\frac{\beta}{4}t} dt + 2 \sin (\alpha_\lambda/2)dB_t, \qquad \alpha_\lambda(0)= 0.
\end{align}
Then $\lim_{t\to \infty} \frac{1}{2\pi} \alpha_\lambda(t)$ exists almost surely, and it has the same distribution as $N_\beta(\lambda).$
\end{theorem}
This describes the one-dimensional marginals of the counting function $N_\beta(\lambda)$. One can also  describe the  distribution of the whole function $N_\beta(\cdot)$ in terms of a one-parameter family of diffusions (see \cite{BVBV}), but we will not need this here. The following proposition summarizes some of the basic properties of the diffusion (\ref{sinesde}). Most of these properties are straightforward, see Proposition 9 in \cite{BVBV} for some of the proofs. 
\begin{proposition}\label{prop:prop}
Let $\lambda>0$.
\begin{enumerate}[(i)]
\item The process $\lfloor \frac{\alpha_\lambda(t)}{2\pi}\rfloor$ is nondecreasing in $t$.  In particular  for any integer $n$ we have $\alpha_\lambda(t)> 2\pi n$ if $t$ is larger than the first hitting time of $2\pi n $.

\item If $0<\lambda<\lambda'$ and $\alpha_\lambda, \alpha_{\lambda'}$ are strong solutions of (\ref{sinesde}) with the same Brownian motion $B_t$ then $\alpha_\lambda(t)\le \alpha_{\lambda'}(t)$ for all $t\ge 0$.

\item For any integer $n$ the process $\alpha_\lambda(t)+2n \pi$ solves the same SDE as (\ref{sinesde}) but  with an initial condition of $2n\pi$ instead of 0. 

\item For any $T>0$ the process $\phi(t)=\alpha_\lambda(T+t)$ solves the same SDE as (\ref{sinesde}) with $\lambda e^{-\frac{\beta}{4}T}$ in place of $\lambda$, $B_{t+T}$ in place of $B_t$ and initial condition $\phi(0)=\alpha_\lambda(T)$. 

\end{enumerate}

\end{proposition}

An important consequence of the previous proposition is the following observation. 
\begin{lemma}\label{lem:rec}
Fix $n>1$ and $\lambda>0$. Let $\tau_\lambda$ be the first hitting time of $2\pi$ of the diffusion $\alpha_\lambda$. (Note that $\tau_\lambda$ might be $\infty$). Then
\begin{align}\label{recursion}
P(N_\beta(\lambda)\ge n)=E\left[ g(\tau_\lambda) \right], \quad \textup{where} \quad g(t)=P\left(N_\beta(\lambda e^{-\frac{\beta }{4} t})\ge n-1\right),\quad g(\infty)=0. 
\end{align}
\end{lemma}
\begin{proof}
According to Proposition \ref{prop:prop}, if we condition on $\{\tau_\lambda=T\}$ with a $T>0$ then the process $\{\alpha_\lambda(t+T)-2\pi, t\ge 0\}$ will have the same distribution as  $\{\alpha_{\lambda e^{-\frac{\beta}{4}T}}(t), t\ge 0\}$. Thus
\begin{align}\label{eqrec}
P(\lim_{t\to \infty} \alpha_\lambda(t)\ge 2\pi n \big| \tau_\lambda=T)=P(\lim_{t\to \infty} \alpha_{\lambda e^{-\frac{\beta}{4}T}}(t)\ge 2\pi (n-1)).
\end{align}
The statement of the lemma now follows by taking expectations in $\tau_\lambda$ and using the description of $N_\beta(\lambda)$ from Theorem \ref{thm:sine}.
\end{proof}
We record the following simple corollary of Lemma \ref{lem:rec}. 
\begin{proposition}\label{prop:trivbnd}
For any $\lambda>0$ and $n\ge 1$ we have
\begin{align}
P(N_\beta(\lambda)\ge n)\le \left(\frac{\lambda}{2\pi}\right)^n.
\end{align}
\end{proposition}
\begin{proof}
From (\ref{sinesde}) one gets $E\alpha_\lambda(t)=\lambda(1-e^{-\frac{\beta}{4}t})$ and using the properties of the diffusion $\alpha_\lambda$ one can show $E \lim_{t\to \infty} \alpha_\lambda(t)=\lambda$. Markov's inequality now implies the statement of the proposition for $n=1$. For larger values we can use Lemma \ref{lem:rec} and the monotonicity of $g(t)$ from (\ref{recursion}) to get
\[
P(N_\beta(\lambda)\ge n)\le P(\tau_\lambda<\infty) P(N_\beta(\lambda)\ge n-1)\le P(N_\beta(\lambda)\ge 1)P(N_\beta(\lambda)\ge n-1)
\]
from which the statement follows by induction. 
\end{proof}

Our proof of Theorem \ref{thm:overcrowding} will rely on the recursion equation (\ref{recursion}) and estimates on $P(\tau_\lambda<t)$. We will show that if we can get a good enough upper or lower bound on $P(N_\beta(\lambda)\ge n-1)$ for all $\lambda\le \lambda_0$, then we can also get  good bounds for $P(N_\beta(\lambda)\ge n)$ for all $\lambda\le \lambda_0$.
 The next section will contain the necessary hitting time estimates  while the remaining sections  will contain the proofs of the upper and lower bounds in Theorem \ref{thm:overcrowding}.

\section{Hitting time estimates}

We will estimate the hitting time $\tau_\lambda$ by coupling the diffusion $\alpha_\lambda(t)$ of (\ref{sinesde})  to a similar diffusion with a constant drift term. Let $\tilde \alpha_\lambda$ be the strong solution of the SDE
\begin{equation}\label{sdetilde}
d \tilde \alpha_\lambda = \lambda dt + 2 \sin \left( \frac{\tilde \alpha_\lambda}{2} \right) dB_t, \qquad \tilde \alpha_\lambda(0) = 0,
\end{equation}
using the same Brownian motion as (\ref{sinesde}). Denote by $\tilde \tau_\lambda$ the first hitting time of $2\pi$ by $\tilde \alpha_\lambda$. Simple coupling arguments show that 
$\alpha_\lambda(t)\le \tilde \alpha_{\lambda'}(t)$ for all $t\ge 0$ with $\lambda'=\frac{\beta}{4}\lambda$, and $\alpha_\lambda(t)\ge \tilde \alpha_{\lambda''}(t)$ for $0\le t\le T$ if $\lambda''=\frac{\beta}{4}e^{-\frac{\beta}{4}T} \lambda$. This yields
\begin{align}\label{coup1}
P(\tau_\lambda\le t)&\le P(\tilde \tau_{\lambda'}\le t), \qquad  \textup{with }\qquad \lambda'=\frac{\beta}{4}\lambda, 
\quad\textup{ and}\\
P(\tau_\lambda\le t)&\ge P(\tilde \tau_{\lambda''}\le t), \qquad  \textup{with } \qquad \lambda''=\frac{\beta}{4}e^{-\frac{\beta}{4}t} \lambda.\label{coup2}
\end{align}
Our bounds will use certain special functions which we now briefly go over. For $a<1$ we set  
\begin{align}
K(a)=\int_0^{\pi/2} \frac{dx}{\sqrt{1-a\sin^2x}}, \qquad E(a)=\int_0^{\pi/2} \sqrt{1-a \sin^2 x} dx.
\end{align}
These functions are the complete elliptic integrals of the first and second kind. We note the identity 
\begin{align}\label{Kind}
K(-a)=\frac12 \int_{-\infty}^\infty \frac{1}{\sqrt{\cosh^2 z+a}}dt, \qquad a>-1. 
\end{align}
We denote by $\WW$  the lower branch of the Lambert W function (or product log function). This is defined as the unique solution of 
\begin{align}
z=\WW(z) e^{\WW(z)} \qquad \textup{with} \qquad \WW(z)\le -1. 
\end{align}
Note that $\WW$ has domain $[-1/e,0)$, it is strictly decreasing and satisfies $\WW(x\log x)=\log x$ for $0<x\le 1/e$. 

The following proposition summarizes some of the relevant properties of the introduced special functions. 

\begin{proposition}\label{prop:asym}
There exists a constant $c>0$, so that the following bounds hold. 
\begin{align}\label{K}
&\left|K(-a)-\frac{1}{2\sqrt{a}}\log(16 a)\right|\le \frac{c}{a^{3/2}}\log(a), \quad \left|E(-a)-\sqrt{a} \right|\le \frac{c}{a^{1/2}}\log(a),\quad \textup{if $a>2$.} 
\\[5pt]
&\label{Kinv}
\hskip100pt\left|K^{-1}(x)+x^{-2} \WW^2(-x/4)  \right|\le c,\qquad \textup{if  $0<x<1/2$.}
\\[5pt]
&\label{W}
\left| \WW(- x) - \WW(-y) \right| \leq c \left|\log(x/y)\right|, \qquad | \WW(-x) | \leq c \log\left(\tfrac1{x}\right), \qquad\textup{if  $x,y \in (0, \tfrac{1}{2e}]$.}
\end{align}
\end{proposition}
\begin{proof}[Outline of the proof] The estimates (\ref{K}) are from the Appendix of \cite{HV}. The definition of $\WW$ allows us to check that the inverse of the function $\frac{1}{2\sqrt{a}}\log(16 a)$ on $[e^2/16,\infty)$ is exactly $x^{-2}\WW^2(-x/4)$. From this fact and the first bound in (\ref{K}) the proof of (\ref{Kinv}) is a simple exercise. Finally, (\ref{W}) follows from the observation that $\WW'(x)=\frac{1}{x}\cdot \frac{\WW(x)}{1+\WW(x)}$. 
\end{proof}

Now we are ready to state the estimates on the hitting times $\tau_\lambda$ and $\tilde \tau_\lambda$. 

\begin{proposition}
\label{prop:taubnd}
There exist positive constants $\eps$ and $c$  (depending only on $\beta$) so that if $\lambda t<\eps$ then
\begin{equation}
\label{tauUB}
P( \tau_\lambda < t) \le  \exp \left[ - \frac{2}{t} \WW^2( -\lambda t) + c\left(1+ \frac{1}{t}\right) \log ( \tfrac{1}{\lambda t})\right].
\end{equation}
For any $\lambda_0>0$ there exist positive constants $\eps, c$ (depending only on $\lambda_0, \beta$) so that if  $0<\lambda\leq \lambda_0$,  $\lambda t< \eps$ and $t \leq \eps \log (\frac{1}{\lambda t})$ then
\begin{equation}
\label{tauLB}
P( \tau_\lambda < t) \geq  \exp \left[ - \frac{2}{t} \WW^2( -\lambda t) - c\left(1+ \frac{1}{t}\right) \log ( \tfrac{1}{\lambda t})\right].
\end{equation}
Moreover, the same statements hold for the hitting time $\tilde \tau_\lambda$ (with possibly different constants $\eps$ and $c$). 
\end{proposition}

We will prove the upper and lower bounds separately. In both cases we will prove the estimate for $\tilde \tau_\lambda$ first, then use (\ref{coup1}) and (\ref{coup2}) to get the corresponding statement for $\tau_\lambda$. 

\subsection{Proof of the lower bounds in Proposition \ref{prop:taubnd}}

By introducing the new variable $X_\lambda=\log(\tan(\tilde \alpha_\lambda/4))$ we can transform (\ref{sdetilde}) into the SDE 
\begin{equation}\label{sde_X}
dX_\lambda = \frac{\lambda}{2} \cosh X_\lambda \ dt + \frac{1}{2} \tanh X_t \ dt + dB_t, \qquad X_\lambda(0) = -\infty.
\end{equation}
Under this transformation $\tilde \tau_\lambda$ becomes the hitting time of $\infty$ for the process $X_\lambda$. Following \cite{HV} we compare $X_\lambda$ with another family of diffusions. For $a\ge -1$ consider the SDE
\begin{equation}\label{sde_Y}
dY_{\lambda, a} = \frac{\lambda}{2} \sqrt{ \cosh^2 Y_{\lambda, a} + a } \ dt + \frac{1}{2} \tanh Y_{\lambda, a} \ dt + dB_t, \qquad Y_{\lambda, a}(0) = - \infty,
\end{equation}
and  denote by $\tau_{Y,\lambda,a}$  the hitting time of $\infty$.  In \cite{HV} the Cameron-Martin-Girsanov formula was  used for diffusions with explosions to compare $X_\lambda$ and $Y_{\lambda, a}$ and the following results were obtained.
\begin{proposition}[\cite{HV}]
For $a>-1$ and $s_1<s_2$ we have
\begin{align}\label{CMG1}
P(\tilde \tau_\lambda\in [s_1,s_2])=E\left[ \ind(\tau_{Y,\lambda,a}\in [s_1,s_2]) e^{-G_{\tau_{Y,\lambda,a}}(Y)}  \right].
\end{align}
Here $G_{\tau_{Y,\lambda,a}}(Y)$ is the Girsanov factor for which the following estimate holds:
\begin{align}\label{CMG2}
\left|G_{\tau_{Y,\lambda,a}}(Y)+\frac{\lambda^2a\tau_{Y,\lambda,a}}{8}-\lambda ((1+a)K(-a)-E(-a))    \right|\le \frac{\lambda \sqrt{|a|}}{4} \tau_{Y,\lambda,a}.
\end{align}
\end{proposition}
The proof of these statements are contained in Section 3 of \cite{HV}. Note the slight notational inconsistency,  in the statements in \cite{HV} $-a$ is used in place of $a$. 

Our lower bound proof will rely on the following lemma, which will be proved at the end of this subsection. 
\begin{lemma}\label{lemma:lowerY}
There is a constant $c>0$ so that for any $a, \lambda$ with  $\lambda \sqrt{a}\ge 2$ and $a>2$ we have
\begin{align}
P\left(\lambda \tau_{Y,\lambda,a}\in \left[4K(-a)(1- \tfrac{5}{\lambda \sqrt{a}}), 4 K(-a)(1+\tfrac{5}{\lambda \sqrt{a}})\right]\right)\ge c \sqrt{K(-a)}.
\end{align}
\end{lemma}

Now we are ready to prove the lower bounds in Proposition \ref{prop:taubnd}.

\begin{proof}[Proof of the  lower bounds in Proposition \ref{prop:taubnd}]
We  use (\ref{CMG1}) and (\ref{CMG2}) with an $a>0$   to get
\begin{align*}
P\left(\lambda \tilde \tau_\lambda\le s_2\right)\ge P\left(\lambda \tilde \tau_\lambda \in \left[s_1,s_2\right]\right)&=E\left[ \ind(\lambda \tau_{Y,\lambda,a}\in [s_1,s_2]) e^{-G_{\tau_{Y,\lambda,a}}(Y)}  \right] \\& \ge E\left[ \ind(\lambda \tau_{Y,\lambda,a}\in [s_1,s_2]) e^{\frac{\lambda^2a\tau_{Y,\lambda,a}}{8}-\lambda ((1+a)K(-a)-E(-a)) -\frac{\lambda \sqrt{a}}{4} \tau_{Y,\lambda,a}}  \right].
\end{align*}
Assuming $\lambda \sqrt{a}\ge 2$ we see that the exponential term under the last expectation is nondecreasing in $\tau_{Y,\lambda, a}$ and because of the indicator function we get the bound
\begin{align*}
P\left(\lambda \tilde \tau_\lambda \le s_2\right)&\ge 
P(\lambda \tau_{Y,\lambda,a}\in [s_1,s_2]) e^{\frac{\lambda a s_1}{8}-\lambda ((1+a)K(-a)-E(-a)) -\frac{ \sqrt{a}}{4} s_1 }. 
\end{align*} 
Now choose $[s_1,s_2]=\left[4K(-a)(1- \tfrac{5}{\lambda \sqrt{a}}), 4 K(-a)(1+\tfrac{5}{\lambda \sqrt{a}})\right]$, and assume $a>2$.
With Lemma \ref{lemma:lowerY} we can estimate the probability on the right and  we can use the asymptotics (\ref{K}) of Proposition \ref{prop:asym} to estimate the exponential term. 
Collecting all the terms gives 
\begin{align}\label{bound123}
P\left(\lambda \tilde \tau_{\lambda}\le 4 K(-a)(1+\tfrac{5}{\lambda \sqrt{a}})\right)\ge \exp \Big[ - \frac{\lambda a}{2} K(-a) -C(1+\lambda) \log(a)  \Big]
\end{align}
with an absolute constant $C$.

Now assume that $ \lambda_0>0$ is fixed, choose a small $\hat \eps>0$ and assume that $0<\lambda\le \lambda_0$, $\lambda  T\le \hat \eps$  and $ T\le \hat \eps \log(\tfrac{1}{\lambda  T})$. Set $a=-K^{(-1)}(\frac{\lambda T}{4})$. If $\hat \eps$ is chosen small enough then using the bounds of Proposition \ref{prop:asym} one can readily show that  
\begin{align}
\frac12 \frac{\log(\tfrac{1}{\lambda  T})}{\lambda  T}\le  \sqrt{a}\le 2 \frac{\log(\tfrac{1}{\lambda  T})}{\lambda  T}
\end{align}
which yields both $\lambda \sqrt{a}\ge 2$ and $a> 2$ (using $\lambda\le \lambda_0$) for small enough $\hat \eps$. Thus we can use (\ref{bound123}) to get 
\begin{align}\notag
\log P \left( \tilde \tau_ \lambda  \leq  T \left(1+ \tfrac{10  T}{\log (1/(\lambda  T)}\right)\right)  &\geq - \frac{\lambda a}{2}\cdot \frac{\lambda  T}{4} - C(1+\lambda) \log a\\ \notag
&  \geq - \frac{2}{ T}\WW^2\left(- \tfrac{\lambda  T}{16}\right) - C \log(\tfrac{1}{\lambda  T})\\&\geq - \frac{2}{ T}\WW^2\left(- \lambda  T\right) - C \log(\tfrac{1}{\lambda  T})\label{bound234}
\end{align}
where we again used Proposition \ref{prop:asym} and the constant $C$ can change from line to line (but only depends on $\lambda_0$ and $\hat \eps$). Now set $t= T \left(1+ \tfrac{10  T}{\log (1/(\lambda  T)}\right)$, we can see that $T\le t \le T(1+10 \hat \eps)$. The  bound (\ref{bound234})  gives
\begin{align*}
\log P \left( \tilde \tau_ \lambda  \leq  t \right) & \geq -\frac{2}{t}\WW^2(-\lambda T)(1+ \tfrac{10  T}{\log (1/(\lambda  T))})- C \log(\tfrac{1}{\lambda  T})\\
&=-\frac{2}{t}\WW^2(-\lambda T)-\frac{20T}{t}\WW^2(-\lambda T) \frac{1}{\log (1/(\lambda  T))}-C \log(\tfrac{1}{\lambda  T}).
\end{align*}
Using the second bound of (\ref{W}) (which is allowed since $\lambda T\le \hat \eps$) together with $1\le \frac{t}{T} \le 1+10 \hat \eps$ we get that 
\[
\frac{20T}{t}\WW^2(-\lambda T) \frac{1}{\log (1/(\lambda  T))}+C \log(\tfrac{1}{\lambda  T})\le C' \log(\tfrac{1}{\lambda  t}).
\]
Using the first and then the second bound of (\ref{W})  we get
\[
-\frac{2}{t}\WW^2(-\lambda T)\ge -\frac{2}{t}(\WW^2(-\lambda t)+ C\log \left|\tfrac{T}{t}\right| \cdot \WW(-\lambda t))\ge -\frac{2}{t}\WW^2(-\lambda t)-\frac{C'}{t} \log(\tfrac{1}{\lambda  t}). 
\]
Putting everything together gives 
\[
\log P \left( \tilde \tau_ \lambda  \leq  t \right) \ge  -\frac{2}{t}\WW^2(-\lambda t)-c\left(1+\frac{1}{t}\right) \log(\tfrac{1}{\lambda  t})
\]
with a constant $c$. The only thing left to show is that for small enough $\eps'>0$  if $\lambda t\le \eps'$, $t\le \eps' \log(1/(\lambda t))$ then we can find $T$ with $t= T (1+ \frac{10  T}{\log (1/(\lambda  T))})$which satisfies $\lambda  T\le \hat \eps$  and $ T\le \hat \eps \log(\tfrac{1}{\lambda  T})$. This is fairly straightforward, one can check that if $\eps'$ is small enough then we can find a $T$ with $(1-10\eps')t\le T\le t$ satisfying all the listed conditions.

Finally, to prove (\ref{tauLB}) for $\tau_\lambda$ we use the comparison (\ref{coup2}). 
Note that if we have the bound $\lambda\le \lambda_0$ then $\lambda''=\frac{\beta}{4} e^{-\frac{\beta}{4}t} \lambda\le \lambda_0''=\frac{\beta}{4}\lambda_0$. By choosing a small enough $\eps$ in the bounds $\lambda t\le \eps$ and $t\le \eps \log(\tfrac{1}{\lambda  t})$ we can use the just proved lower bound for $\tilde \tau_{\lambda''}$ to give
\begin{align*}
\log P(\tau_\lambda\le t)\ge -\frac{2}{t} \WW^2(-\lambda \tfrac{\beta}{4}e^{-\frac{\beta}{4}t} t)-c\left(1+\frac{1}{t}\right)\left(\log(\tfrac{1}{\lambda t})+\frac{\beta}{4}t+\log\frac{4}{\beta}\right).
\end{align*}
The fact that the right hand side is bounded from below by $-\frac{2}{t}\WW^2(-\lambda t)-\frac{c'}{t} \log(\tfrac{1}{\lambda  t})$ with a new constant $c'$ follows along the line of the already seen arguments using Proposition \ref{prop:asym} and our assumptions on $\lambda, t$. 
\end{proof}
Now we return to the proof of Lemma \ref{lemma:lowerY}.

\begin{proof}[Proof of Lemma \ref{lemma:lowerY}]
Our main tool is again a coupling argument. We introduce the diffusions
\begin{align*}
d Y_-&=\frac{\lambda}{2}\sqrt{\cosh^2 Y_-+a}\ dt-\frac12  dt +dB\\
d Y_+& = \frac{\lambda}{2}\sqrt{\cosh^2  Y_+ +a}\ dt+\frac12  dt +dB
\end{align*}
driven by the same Brownian motion as $Y$. Comparing the drifts of $Y_{-}, Y_{+}$ to that of $Y$ we see  that $Y_{-}\le Y$ and $Y\le Y_{+}$ until the blowup of the larger diffusions. 

We also introduce $Z_-=Y_-+t/2-B_t$ and $Z_+= Y_+ - t/2 - B_t$ which satisfy the differential equations
\[
Z_-'=\frac{\lambda}{2}\sqrt{\cosh^2(Z_- -t/2+B_t)+a}>0, \quad \text{ and } \quad Z_+'=\frac{\lambda}{2}\sqrt{\cosh^2(Z_+ +t/2+B_t)+a}>0
\]
until their blowup times.  Note that $Z_-$ blows up exactly when $Y_-$ blows up, and $Z_+$ blows up when $Y_+$ does.   Thus  $Z_+$ will blow up before  $Y$ does and $Z_-$ will blow up after.

We introduce the constants 
\[
\eps_1=\frac{10 K(-a)}{\lambda \sqrt{a}}, \qquad \eps_2=\frac{K(-a),}{8\lambda}  \qquad \theta=\frac{4K(-a)+\eps_1}{\lambda},
\]
and we record for further use that since  $\lambda \sqrt{a}\ge 2$ we have
\begin{align}\label{111}
\eps_1>\frac{4\eps_2+\theta}{ \sqrt{a}}.
\end{align}
We will show that on the event $H=\{\sup_{t\le \theta} |B_t|\le \eps_2\}$ the process $Z_-$ will blow up before $\theta$ and the process $Z_+$ will blow up after $\lambda^{-1}(4K(-a)-2\eps_1)$. The final bound will come from an estimate on the probability of the event $H$. 

Note that  $Z'_{-}>\frac{2\sqrt{a}}{\lambda}>0$ until its blowup, in particular $Z_{-}$ is strictly increasing. Moreover, on the event $H$ for $t\le \theta$ we have
\[
Z_{-}(t)-\theta/2-\eps_2\le Z_{-}(t)-t/2+B_t \le Z_{-}(t)+\eps_2.
\]
Let $\eta_1$ be the  hitting time of $-\eps_2$ by $Z_{-}$. Then for $t\le \eta_1 \wedge \theta$ on the event $H$ we have
\[
Z_{-}'(t)=\frac{\lambda}{2}\sqrt{\cosh^2(Z_{-}(t)-t/2+B_t)+a}\ge \frac{\lambda}{2}\sqrt{\cosh^2(Z_{-}(t)+\eps_2)+a}.
\]
Rearranging, integrating and using  identity (\ref{Kind}):
\[
\frac{\lambda}{2}(\eta_1 \wedge \theta) \le \int_0^{\eta_1\wedge \theta} \frac{Z_-'(t)}{\sqrt{\cosh^2(Z_-(t)+\eps_2)+a}}dt\le\int_{-\infty}^{-\eps_2} \frac{1}{\sqrt{\cosh^2(z+\eps_2)+a}}dz=K(-a),
\]
which gives
\[
\lambda (\eta_1\wedge \theta)\le 2K(-a)\leq \lambda \theta, \qquad \textup{and} \qquad \lambda \eta_1 \leq 2K(-a).
\]
Next we bound the travel time of $Z_-$ from $-\eps_2$ to $\theta/2+\eps_2$ which we denote by  $\eta_2$. Since $Z_-'\ge \frac{\lambda}{2}\sqrt{a}$ we have the bound
\[
\lambda \eta_2\le \frac{4\eps_2+\theta}{ \sqrt{a}}\le \eps_1. 
\]
Note that we have $\eta_1+\eta_2\le \lambda^{-1}(2K(-a)+ \eps_1)<\theta$ on $H$. 
Consider now the travel time $\eta_3$ the  from $\theta/2+\eps_2$ to $\infty$.  Denote by $\hat \eta_3$ the minimum of $\eta_3$ and $\theta-\eta_1-\eta_2>0$.
On the event $H$ if $Z_- \geq \theta/2 + \eps_2$, $\eta_1+ \eta_2 \leq  t \leq \theta$ then we have the lower bound
\[
Z_-'(t)=\frac{\lambda}{2}\sqrt{\cosh^2(Z_-(t)-t/2+B_t)+a}\ge \frac{\lambda}{2}\sqrt{\cosh^2(Z_-(t)- \theta/2-\eps_2)+a}.
\]
Rearranging the inequality and integrating from $\eta_1+ \eta_2$ to $\eta_1+ \eta_2+ \hat \eta_3$ we get 
\[
\lambda \hat\eta_3\le 2 K(-a).
\]
Collecting our estimates we see that on the event $H$ we have
\[
\lambda ( \eta_1+\eta_2 + \hat \eta_3 )\le 4 K(-a)+\eps_1=\lambda \theta.
\]
This immediately implies $\hat \eta_3=\eta_3$ and since $\eta_1+\eta_2+\eta_3$ is the blowup time of $Y_{-}$ we also get that $ \tau_{Y,\lambda, a} \le \theta$ on $H$.

To show that the blow up time of $Z_+$ is greater than $\lambda^{-1}(4K(-a)-2\eps_1)$ on the event $H$ we carry out a similar calculation.  The process $Z_+$ is strictly increasing and on the event $H$ for $t\le \theta$ we have
\[
Z_{+}(t)-\eps_2\le Z_{+}(t)+t/2+B_t \le Z_{-}(t)+\theta/2+\eps_2.
\]
Let $\zeta_1$ be the hitting time of $-\theta/2-\eps_2$ for process $Z_+$. Since the blowup time of $Z_+$ is smaller than that of $Z_-$ we get that $\zeta_1<\theta$ on the event $H$.  For $Z_+ < -\theta/2-\eps_2$ on $H$ we get
\[
Z_+'(t)=\frac{\lambda}{2}\sqrt{\cosh^2(Z_+ +t/2+B_t)+a} \leq \frac{\lambda}{2}\sqrt{\cosh^2(Z_+(t)- \eps_2)+a},
\]
which gives us 
\begin{align*}
\frac{\lambda}{2}\zeta_1  &\geq \int_0^{\zeta_1} \frac{Z'_+(t)}{\sqrt{\cosh^2(Z_+(t)- \eps_2)+a}} dt  =K(-a)- \int_{-\theta/2-2\eps_2}^{0}\frac{1}{\sqrt{\cosh^2 (z)+a}}dz\\& \geq K(-a) - \frac{\theta+4\eps_2}{2\sqrt{a}}\geq K(-a)-\eps_1/2.
\end{align*}
In the last step we used (\ref{111}). 
Now take $\zeta_3$ to be the time spent by $Z_+$ traveling from $\eps_2$ to $\infty$.  Again note that the blow up time of $Z_+$ is bounded by that of $Z_-$ which is bounded by $\theta$ on $H$.  Therefore on $H$ for $\eps_2\le Z_+$ we have 
\[
Z_+'(t)=\frac{\lambda}{2}\sqrt{\cosh^2(Z_+(t)+t/2+B_t)+a}\le \frac{\lambda}{2}\sqrt{\cosh^2(Z_+(t)+ \theta/2+\eps_2)+a}.
\]
This yields
\[
\frac{\lambda}{2} \zeta_3 \geq  K(-a) - \int_{0}^{\theta/2+2\eps_2} \frac{1}{\sqrt{\cosh^2(z)+a}}dz \geq K(-a) - \eps_1/2.
\]
These estimates give
\[
\lambda( \zeta_1+ \zeta_3) \geq 4K(-a)-2\eps_1.
\]
Our last bound shows that on $H$ the process $Y$ blows up after $\lambda^{-1}(4K(-a)-2\eps_1)$. 
Putting everything together yields
\[
P\left(\lambda \tau_{Y,\lambda,a}\in \left[4K(-a)(1- \tfrac{5}{\lambda \sqrt{a}}), 4 K(-a)(1+\tfrac{5}{\lambda \sqrt{a}})\right]\right)\ge P(H). 
\]
On the other hand we have
\[
P(H)=P(\sup_{t\le \lambda^{-1} \theta} |B_t|\le \eps_2)=P(\sup_{t\le 1}|B_t|\le \lambda^{1/2} \theta^{-1/2} \eps_2)
\]
and
\[
 \lambda^{1/2} \theta^{-1/2} \eps_2= \frac{K(-a)}{8 \sqrt{4K(-a) (1+\tfrac{10}{\lambda \sqrt{a}})}}\ge \frac{1}{40}\sqrt{K(-a)}.
\]
It is known that $P(\sup_{t\le 1}|B_t|\le \delta)\ge 2(2\Phi(\delta/4)-1)$ for $\delta>0$ (see e.g.~Proposition 2.8.10 in \cite{KarShr}) which can be bounded from below by $c \delta$ if we have an upper bound on $\delta$. Since  $a>2$ we have an upper bound on $ \frac{1}{40}\sqrt{K(-a)}$ which finishes the proof.
\end{proof}

\subsection{Proof of the upper bounds in Proposition \ref{prop:taubnd}}

The proof of the upper bound will rely on the following estimate on the moment generating function of $\tilde \tau_\lambda$.
\begin{proposition}[\cite{HV}]\label{prop:taumgf}
For any $a>-1$ we have
\begin{align}\label{CMG3}
E e^{-\frac{\lambda^2a}{8}\tilde \tau_{\lambda}-\frac{\lambda \sqrt{|a|}}{4}\tilde \tau_{\lambda}}\le e^{-\lambda  ((1+a)K(-a)-E(-a))}.
\end{align}
\end{proposition}
This is Proposition 8 in \cite{HV}, note that in that statement $-a$ is used in place of $a$. 
\begin{proof}[Proof of the upper bound in Proposition \ref{prop:taubnd}]
By  Markov's inequality we have 
\[
P(\tilde\tau_\lambda \leq t) \leq \big( E e^{A\tilde \tau_\lambda}\big)e^{-At}\quad \textup{ for $A<0$}.
\]  We will choose $A = -\frac{\lambda^2 a}{8} - \frac{\lambda \sqrt{a}}{4}$ with $a= -K^{(-1)}(t\lambda/4)\ge -K^{(-1)}( \eps/4)>0$. Using Proposition \ref{prop:taumgf} we get
\[
Ee^{A\tilde \tau_\lambda}\le e^{-\lambda ((1+a)K(-a)-E(-a))}.
\]
This gives the upper bound
\begin{align*}
P(\tilde \tau_\lambda \le t) &\leq e^{-\lambda ((1+a)K(-a)-E(-a)) + \frac{\lambda^2 a}{8}t+ \frac{\lambda \sqrt{a}}{4} t}
\leq e^{K^{(-1)}(t\lambda/4) \frac{\lambda^2 t}{8}+\lambda E(-a)+\frac{\lambda \sqrt{a}t}{4}},
\end{align*}
where we used our specific choice of $a$. If $ \eps$ is small enough we will have $a\ge -K^{(-1)}( \eps/4)>2$ which allows us to use the asymptotics of Lemma \ref{prop:asym}. This gives the bound
$
a\le \frac{C}{\lambda^2 t^2}\log^2(\frac{1}{\lambda t})
$
with a positive constant $C$, and also
\begin{align*}
\frac{K^{(-1)}(t\lambda/4) \lambda^2 t}{8}+\lambda E(-a)+\frac{\lambda \sqrt{a}t}{4}\le - \frac{2}{t} \WW(-\lambda t/4)^2+\frac{C}{t}+C(1+\frac{1}{t}) \log(\tfrac{1}{\lambda t})
\end{align*}
(with a possibly different $C$). Using the bounds in (\ref{W}) and changing the value of $C$ again we get 
\[
\frac{K^{(-1)}(t\lambda/4) \lambda^2 t}{8}+\lambda E(-a)+\frac{\lambda \sqrt{a}t}{4}\le- \frac{2}{t} \WW(-\lambda t)^2+C(1+\frac{1}{t}) \log(\tfrac{1}{\lambda t}),
\]
which finishes the proof for $\tilde \tau_\lambda$.

To get the similar statement for $\tau_\lambda$ we use (\ref{coup1}) to obtain
\begin{align}
P(\tau_\lambda\le t)\le P(\tilde \tau_{\frac{\beta}{4}\lambda}\le t)\le  \frac{2 }{ t}\WW^2 \left(-\frac{\lambda\beta}{4} t\right)  + c(1+ \frac{1}{t})\log \left( \tfrac{1}{\frac{\lambda \beta}{4} t}\right).\label{last}
\end{align}
Note that we are allowed to use the just proved bound on $\tilde \tau$ if we choose the upper bound $\eps$ on  $\lambda t$ small enough. To complete the proof we use (\ref{W}) to replace  $\frac{\lambda \beta}{4}$  with $\lambda$ everywhere by choosing $\eps$ small enough and modifying the constant in front of the last term of (\ref{last}).  
\end{proof}

\section{Proof of the lower bound in Theorem \ref{thm:overcrowding}}

Our proof will have two main steps. In the first step we will give a reasonably good (but not optimal) lower bound on $P(N_\beta(\lambda)\ge n)$. In the second step we will show how one can get better and better bounds as we increase $n$.

\begin{proof}[Proof of the lower bound in Theorem \ref{thm:overcrowding}]
We first give a direct estimate on $P(N_\beta(\lambda)\ge n)$ for $\lambda\le \lambda_0$. Although this lower bound will be far from the statement of the theorem, we will later show how to `bootstrap' this into a near optimal bound for large $n$.

Recall Theorem \ref{thm:sine} and Proposition \ref{prop:prop}. The connection between the counting function $N_\beta(\lambda)$ and the diffusion $\alpha_\lambda(t)$ shows that for any $T>0$ we have
\[
P(N_\beta(\lambda)\ge n)\ge P(\alpha_\lambda(T)\ge 2n\pi). 
\]
The drift in (\ref{sinesde}) on $[0,T]$ is at least $\lambda'=\frac{\beta}{4}e^{-\frac{\beta}{4}T} \lambda$ so coupling $\alpha_\lambda$ and $\tilde \alpha_{ \lambda'}$ (using the same driving noise) we get
\[
P(\alpha_\lambda(T)\ge 2n\pi)\ge P(\tilde \alpha_{\lambda'}(T)\ge 2n\pi).
\]
For the process $\tilde \alpha_{\lambda'}$ denote the hitting time of $2\pi k$ by $\tilde \tau_{\lambda',k}$ (note that $\tilde \tau_{\lambda', 1}=\tilde \tau_{\lambda'}$ using our previous notation). Using the same ideas needed for Proposition \ref{prop:prop} it is easy to show that $\tilde \tau_{\lambda',1}, \tilde \tau_{\lambda', 2}-\tilde \tau_{\lambda', 1}, \dots, \tilde \tau_{\lambda',k}-\tilde \tau_{\lambda',k-1}, \dots$ are i.i.d.~random variables. Using this we can obtain the estimate
\begin{align*}
P(\tilde \alpha_{\lambda'}(T)\ge 2n\pi)&\ge P(\tilde \tau_{\lambda',1}\le \tfrac{T}{n}, \textup{ and } \tilde \tau_{\lambda',k}-\tilde \tau_{\lambda',k-1}\le \tfrac{T}{n} \textup{ for } 2\le k\le n)\\
&=P(\tilde \tau_{\lambda'}\le \tfrac{T}{n})^n
\end{align*}
We will set $T=\frac{4}{\beta}\log(\frac{n}{\lambda})$ and apply the lower bound (\ref{tauLB}) of Proposition \ref{prop:taubnd} with $t=\frac{T}{n}$ and $\lambda'$. 
(Note that  $\lambda'\le \frac{\beta}{4}\lambda_0$.)
We first have to check that the conditions of the proposition are satisfied. From the definitions we get
\[
\lambda' t=\left(\frac{\lambda}{n}\right)^2 \log(\tfrac{n}{\lambda}), \qquad \frac{t}{\log(\frac{1}{\lambda' t})}=\frac{1}{ n}\cdot \frac{4 \log(\frac{n}{\lambda})}{2\beta \log(\frac{n}{\lambda})-\beta \log\log(\frac{n}{\lambda})}.
\]
Since we have $\lambda'\le \frac{\beta}{4}\lambda_0$, there is a positive integer $n_0$ (depending only on $\lambda_0$,  $\beta$) so that if $n\ge n_0$ then both of these expressions are bounded by $\eps$. (Here $\eps$ is the constant from the statement of Proposition \ref{prop:taubnd} which also only depends on $\lambda_0, \beta$.)
Thus if $n$ is large enough we can use Proposition \ref{prop:taubnd} to get
\begin{align*}
P(\tilde \tau_{\lambda'}\le \tfrac{T}{n})&\ge \exp\left[-\frac{2}{t}\WW^2(-\lambda' t)-c\left(1+\frac{1}{t}\right)\log(\tfrac{1}{\lambda' t})    \right]\\
&=\exp\left[-\frac{\beta n}{2 \log(\frac{n}{\lambda})} \WW^2\left(\frac{\lambda^2}{n^2} \log(\tfrac{n}{\lambda})\right)-c\left(1+\frac{\beta n}{4 \log(\frac{n}{\lambda})} \log\left(\frac{n^2}{\lambda^2}\cdot \frac1{\log(\frac{n}{\lambda})}\right)\right)  \right].
\end{align*}
Choosing $n$ large enough we can use the bounds (\ref{W}) of Proposition \ref{prop:asym} to get 
\begin{align*}
P(\tilde \tau_{\lambda'}\le \tfrac{T}{n})\ge \exp\left[-c_1 n \log(\tfrac{n}{\lambda})-c_2 n\right]
\end{align*}
with positive constants $c_1, c_2$ (which only depend on $\lambda_0, \beta$). Putting our estimates together we obtain the lower bound
\begin{align}\label{stupidbnd}
P(N_\beta(\lambda)\ge n)\ge \exp\left[-c_1 n^2 \log(\tfrac{n}{\lambda})-c_2 n^2\right],
\end{align}
which holds for $n\ge n_0$ with $n_0$ only depending on $\lambda_0, \beta$. Finally, since $P(N_\beta(\lambda)\ge k)$ is decreasing in $k$, we can adjust the constants $c_1, c_2$ so that (\ref{stupidbnd}) remains valid for all $n\ge 1$.

To finish the proof we will show how to improve this lower bound to get the lower bound in Theorem \ref{thm:overcrowding}. From the recursion equation of Lemma \ref{lem:rec} we get that  for any $T>0$ and $n>1$
\begin{align}\notag
P(N_\beta(\lambda)\ge n)&\ge E\left[\ind(\tau_\lambda\le T) g(\tau_\lambda)   \right]\\&\ge P(\tau_\lambda\le T) g(T)= P(\tau_\lambda\le T) P(N_\beta(\lambda e^{-\frac{\beta}{4} T})\ge n-1),\label{lwr1}
\end{align}
where we used the fact that the $g(t)$ function is non increasing. 
Set $T= \frac{4 \log(\frac{n}{\lambda})}{\beta n}$. Using the same arguments as in the first part of this proof we can check that there is a positive integer $n_0$ (depending only on $\lambda_0, \beta$) so that for $n\ge n_0$ we can use the lower bound of Proposition \ref{prop:taubnd} with $t=T$. This gives
\begin{align*}
\log P(\tau_\lambda\le T) \ge -\frac{\beta n}{2 \log(\frac{n}{\lambda})} \WW^2\left(-\frac{4}{\beta}\cdot \frac{\lambda}{n}\log(\tfrac{n}{\lambda})\right)-c(1+\frac{\beta n}{4 \log(\frac{n}{\lambda})} )\log\left(\frac{\beta n}{4\lambda} \cdot \frac{1} {\log(\frac{n}{\lambda})}\right).
\end{align*}
Assuming that $n$ is large enough we can use the  bounds in (\ref{W}) to remove the constant $\frac{4}{\beta}$ from inside $\WW$. This allows us to find positive constants $c_1$ and $c_2$ so that
\begin{align}\notag
\log P(\tau_\lambda\le T)& \ge -\frac{\beta n}{2 \log(\frac{n}{\lambda})} \WW^2\left(-\frac{\lambda}{n}\log(\tfrac{n}{\lambda})\right)-c_1 \log(\tfrac{n}{\lambda})-c_2 n\\
&=-\frac{\beta}{2}n \log(\tfrac{n}{\lambda})-c_1\log(\tfrac{n}{\lambda})-c_2 n\label{lwr2}
\end{align}
where we used $\WW(x \log x)=\log x$ which holds for $0<x<1/e$.

Now assume that for a certain $n$ there is a constant $f_{n-1}>0$ so that  we have the lower bound 
\[
P(N_\beta(\lambda)\ge n-1)\ge  \exp\left[-f_{n-1}  \log(\frac{n-1}{\lambda})-c (n-1)^2\right]
\]
for all $\lambda\le \lambda_0$ with a fixed constant $c$. Then
\begin{align*}
\log P(N_\beta(\lambda e^{-\frac{\beta}{4} T})\ge n-1)&\ge -f_{n-1} \log(\tfrac{n-1}{\lambda e^{-\frac{\beta}{4} T}})-c(n-1)^2\\
&=-f_{n-1}\log(\tfrac{n}{\lambda})-f_{n-1}\log(\frac{n-1}{n})-\frac{\beta}{4}T f_{n-1}-c(n-1)^2\\
&\ge -\frac{n+1}{n} f_{n-1}\log(\frac{n-1}{n})-c(n-1)^2. 
\end{align*}
Together with (\ref{lwr1}) and (\ref{lwr2}) this yields the lower bound
\[
\log P(N_\beta(\lambda)\ge n) \ge -\left(\frac{n+1}{n}f_{n-1}+\frac{\beta}{2}n+c_1\right)\log(\tfrac{n}{\lambda})-c_2n-c(n-1)^2.
\]
If we also assume that $c>c_2$ then we can further simplify this to 
\[
\log P(N_\beta(\lambda)\ge n) \ge -\left(\frac{n+1}{n}f_{n-1}+\frac{\beta}{2}n+c_1\right)\log(\tfrac{n}{\lambda})-cn^2.
\]
Thus we proved the following statement. Suppose that $n_0$ is large enough,  $c>c_2$  and $f$ solves the following recursion:
\begin{align}\label{reclwr}
f_{n}=\frac{n+1}{n}f_{n-1}+\frac{\beta}{2}n+c_1, \qquad \textup{for $n> n_0$}.
\end{align}
Then if
\[
\log P(N_\beta(\lambda)\ge n) \ge -f_{n} \log(\tfrac{n}{\lambda})-c n^2
\]
holds for $n=n_0$ then it also holds for any $n>n_0$. Because of (\ref{stupidbnd}), we can choose $f_{n_0}$ so that the previous inequality holds for $n_0$ (with a suitable $c$). To finish the proof, we need to show that the recursion (\ref{reclwr}) will have the appropriate asymptotics. The general solution of the recursion (without the initial condition) is 
\[
f_{n}=\frac{\beta}{2}n(n+1)+(c_1-\frac{\beta}{2})\sum_{k=1}^{n+1} \frac{n+1}{k}+C (n+1).
\]
From this it is clear that there is a solution with any initial condition, and the solution satisfies
\[
f_{n}\le \frac{\beta}{2}n^2+b\, n \log (n+1)
\]
with a finite $b$. From this the lower bound of Theorem \ref{thm:overcrowding} follows for $n\ge n_0$ and $\lambda\le \lambda_0$. Using the argument following  (\ref{stupidbnd}) we can now modify the constant $b$ so that the lower bound holds for all $n$ which completes the proof.
\end{proof}

\section{Proof of the upper bound in Theorem \ref{thm:overcrowding}}

The upper bound will use a two step strategy similar to that of the lower bound, but our recursive step will be a bit more involved.

\begin{proof}[Proof of the upper bound in Theorem \ref{thm:overcrowding}]

Recall from Proposition \ref{prop:trivbnd} that 
\[
P( N_\beta(\lambda) \geq n )\le \left(\frac{\lambda}{2\pi}\right)^n\le \exp(-n \log(1/\lambda)).
\]
By choosing  $c_1>1$ we get $-n \log(1/\lambda)\le -n \log(n/\lambda)+c_1 n^2$ which yields
\begin{align}\label{trivup}
\log P( N_\beta(\lambda) \geq n )\le -n \log(n/\lambda)+c_1 n^2.
\end{align}

Now assume that $n\ge 2$ and there are positive constants $f_{n-1}\le \frac{\beta}{2}(n-1)^2$, $c_1$ so that for all $0\le \lambda\le \lambda_0$ we have
\begin{align}\label{ind}
 P(N_\beta(\lambda)\ge n-1)\le \exp\left[-f_{n-1} \log\left(\tfrac{n-1}{\lambda}\right)+c_1 (n-1)^2\right].
\end{align} 
Our goal is to show that we can choose an $n_0$ (depending only on $\lambda_0, \beta$) so that if $n\ge n_0$ and $c_1$ is larger than a fixed constant then we can get an upper bound of a similar form for $P(N_\beta(\lambda)\ge n)$ as well.

Let $T_0=0<T_1<\dots<T_{M}$. Then
\begin{align}
P(N_\beta(\lambda)\ge n)
&= \sum_{i=0}^{M-1} P(N_\beta(\lambda)\ge n, T_i\le \tau_\lambda\le T_{i+1})
+P(N_\beta(\lambda)\ge n, \tau_\lambda\ge  T_{M}).\label{sum123}
\end{align}
Choose real numbers $\gamma_1< \dots< \gamma_M$ with  $\gamma_1 = 1/3, \gamma_M = 3$, $M\leq 3n$ and $\gamma_{i+1} - \gamma_i \leq 1/n$.  We  set $T_i= \gamma_i \frac{4\log(n/\lambda)}{\beta n}$ for $i=1,\dots,M$, if $n\ge \lambda_0$ then these are positive numbers.
We  will estimate the sum (\ref{sum123}) term by term.

Using the identity (\ref{eqrec})  from the proof of Lemma \ref{lem:rec} we can show that 
\begin{align}
P(N_\beta(\lambda) \geq n , T_i \leq \tau \leq T_{i+1}) & \leq P(N_\beta(\lambda e^{-\frac{\beta}{4} T_i}) \geq n-1)P(\tau < T_{i+1}).
\end{align}
From the definition of $T_i$ (if $n>\max(1,\lambda_0)$) we get
\begin{align}\label{bnd3}
  (1+ \gamma_i\frac{1}{n}) \log(n/\lambda)  - \frac{1}{n-1} \le \log \left( \frac{n-1}{\lambda e^{-\frac{\beta}{4}T_i}}\right).
 \end{align}
Note  that $\lambda T_i \leq \frac{12}{\beta}\frac{\lambda}{n} \log (n/\lambda)$ which can be made as small we wish by choosing $n_0$ large enough. 
Thus we can apply the upper bound in  Proposition \ref{prop:taubnd}  to get
\begin{align}\notag
\log P(\tau<T_{i+1})&\le  - \frac{2}{T_{i+1}} \WW^2( -\lambda T_{i+1}) + c\left(1+ \frac{1}{T_{i+1}}\right)\log (\tfrac{1}{\lambda T_{i+1}}) \\
& \notag\leq - \frac{\beta}{2\gamma_{i+1}}   \frac{n}{\log(n/\lambda)} \WW^2\left( \frac{\lambda}{n} \log \left( \tfrac{\lambda}{n}\right) \right) + C \left( 1 + \frac{\beta}{4\gamma_{i+1}} \frac{n}{ \log(n/\lambda)}  \right) \log (\tfrac{n}{\lambda})\\
& \leq - \frac{\beta}{2\gamma_{i+1}}  n \log (\tfrac{n}{\lambda}) + C \log (\tfrac{n}{\lambda}) + C n.
\end{align} 
Here we used the bounds  (\ref{W}),  $\gamma_{i+1} \ge 1/3$  and the  identity $\WW(x \log x) = \log x$.  The constant $C$ depends only on $\beta$ and $\lambda_0$, but may change from line to line.

Using the assumption (\ref{ind}) for $\lambda e^{-\frac{\beta}{4} T_i}\le \lambda_0$ we get 
\[
P(N_\beta(\lambda e^{-\frac{\beta}{4} T_i}) \geq n-1)\le \exp \left[- f_{n-1} \log \left( \tfrac{n-1}{\lambda e^{-\frac{\beta}{4} T_i}} \right)  +c_1 (n-1)^2\right].
\]
Using (\ref{bnd3}) we can write
\begin{align*}
-f_{n-1} \log \left( \frac{n-1}{\lambda e^{-\frac{\beta}{4}T_i}}\right) &\leq  -f_{n-1}(1+\frac{\gamma_i}{n}) \log (n/\lambda) +\frac{f_{n-1}}{n-1}\\
&\le -\left(f_{n-1}(1+\frac{\gamma_{i+1}}{n}) -\frac{\beta}{2}\right)\log (n/\lambda)+\frac{\beta}{2}(n-1),
\end{align*}
where we used $f_{n-1}\le \frac{\beta}{2}(n-1)^2$ and $\gamma_{i+1}-\gamma_i\leq \frac{1}{n}$, which holds for $i\ge 1$.  We also have
\[
 \log \left( \frac{1}{\lambda  T_{i+1}}\right) = \log(n/\lambda) - \log \log (n/\lambda) - \log (4\gamma_{i+1}/\beta) \leq \log(n/\lambda),
\]
where we used $1/3\le \gamma_i\le 3$ and that $n/\lambda$ can be made sufficiently large  by choosing $n_0$ appropriately.

Collecting all our estimates we get that for $1\le i\le M-1$ we have
\begin{align*}
\log P(N_\beta(\lambda)& \geq n  , T_i \leq \tau \leq T_{i+1}) \\
& \leq  -\left(f_{n-1}(1+\frac{\gamma_{i+1}}{n}) +\frac{\beta}{2\gamma_{i+1}}n\right) \log(n/\lambda) +c_1(n-1)^2 + C \log (\tfrac{n}{ \lambda }) + C n.
\end{align*}
Assuming $c_1 > C$ this yields
\begin{align}\notag
\log P(N_\beta(\lambda) \geq n& , T_i \leq \tau \leq T_{i+1})\\
& \leq \notag -\left(f_{n-1}(1+\frac{\gamma_{i+1}}{n}) +\frac{\beta}{2\gamma_{i+1}}n-C\right) \log(n/\lambda) +c_1 n^2 - c_1n \\
&\le -\left(f_{n-1}+\sqrt{2\beta f_{n-1}} -C\right) \log(n/\lambda) +c_1 n^2 - c_1n \label{1st_term}.
\end{align}
In the last step we used that $x(1+\frac{\gamma}{n})+\frac{\beta}{2\gamma}n\ge x+\sqrt{2\beta x}$ if  $x>0, \gamma>0$.

For the $i=0$ case we can use a similar argument (and the fact that $\gamma_1=\tfrac13$) to show that 
\[
\log P(N_\beta(\lambda) \geq n , T_0 \leq \tau \leq T_{1}) \leq -(f_{n-1}+\frac{\beta}{2} 3n- C)\log(n/\lambda)+c_1n^2 - c_1 n .
\]
Since $f_{n-1}\le \frac{\beta}{2}(n-1)^2$ we have $\frac{\beta}{2}3n \geq \sqrt{2 \beta f_{n-1}}$ and so  inequality (\ref{1st_term}) holds for $i=0$ as well.

To  bound the last term in (\ref{sum123}) we start with the bound
\[
P(N_\beta(\lambda)\ge n, T_{M}\le \tau) \le P(N_\beta(\lambda e^{-\frac{\beta}{4} T_M})\ge n-1).
\]
Recalling that $\gamma_M=3$ and  using arguments similar to what we used to obtain (\ref{1st_term}) we get
\begin{align}
P(N_\beta(\lambda)\ge n, T_{M}\le \tau)  \le -f_{n-1}\left(1+\frac{3}{n}\right) \log(n/\lambda)+c_1n^2 - c_1n\label{2nd_term}.
\end{align}
Applying the upper bounds  (\ref{1st_term})  and (\ref{2nd_term}) on the respective terms of the sum in (\ref{sum123}) yields
\begin{align*}
\log P(N_\beta(\lambda)\ge n)\le&-\min\left(f_{n-1}+\sqrt{2\beta f_{n-1}} -c_2,f_{n-1}\left(1+\frac{3}{n}\right)\right)\log(n/\lambda)\\&\qquad +c_1n^2 - c_1n + \log(3n+1),
\end{align*}
with  constants $c_1, c_2$ that only depend on $\lambda_0, \beta$.  
By choosing $c_1>2$ we can simplify this as
\begin{align*}
\log P(N_\beta(\lambda)\ge n)\le -\min\left(f_{n-1}+\sqrt{2\beta f_{n-1}} -c_2,f_{n-1}\left(1+\frac{3}{n}\right)\right)\log(n/\lambda)+c_1 n^2.
\end{align*}
Combining this result  with (\ref{trivup}) we get the following statement. There is a positive integer $n_0'$, and positive constants $c, c_1$ (all depending only on $\lambda_0, \beta$) so that if we set $f_{n_0}=n_0$ for an $n_0>n_0'$ and
\begin{align}\label{uprec}
f_{n}=\min\left(f_{n-1}+\sqrt{2\beta f_{n-1}} -c_2,f_{n-1}\left(1+\frac{3}{n}\right), \frac{\beta}{2} n^2 \right), \qquad \textup{for } n>n_0,
\end{align}
then
\[
\log P(N_\beta(\lambda)\ge n)\le -f_{n} \log\left(\tfrac{n}{\lambda}\right)+c_1 n^2
\]
for all $n\ge n_0$, $\lambda\le \lambda_0$.

To finish our proof we just have to show that we can choose a large enough $n_0$ as our starting point so that the solution of the recursion (\ref{uprec}) satisfies $f_{n}\ge \frac{\beta}{2}n^2-c n \log(n+1)$ with a constant $c$. 

This is fairly straightforward although somewhat technical. A simple computation shows that if we choose $n_0$ large enough then the following statements hold:
\begin{enumerate}
\item[(i)] $f_{n}$ is strictly increasing, 

\item[(ii)] $f_{n}=f_{n-1}\left(1+\frac{3}{n}\right)$ for $n_0<n<n_1$ for a certain $n_1$, 

\item[(iii)] $ f_{n}=\min\left(f_{n-1}+\sqrt{2\beta f_{n-1}} -c_2,\frac{\beta}{2} n^2 \right)$ for $n\ge n_1$, and

\item[(iv)] if $n$ is large enough and $f_{n-1}\ge \frac{\beta}{2}(n-1)^2-c  (n-1)\log n$ with a large enough $c$ then we also have
\[
f_{n-1}+\sqrt{2\beta f_{n-1}} -c\ge \frac{\beta}{2}n^2-c n\log (n+1).
\]

\end{enumerate}
Putting these statements together we get that $f_{n}\ge \frac{\beta}{2}n^2-c n \log(n+1)$  for a large enough $n$ which gives
\begin{align}\label{end}
\log P(N_\beta(\lambda)\ge n)\le -\left( \frac{\beta}{2}n^2-c n \log(n+1)  \right) \log\left(\tfrac{n}{\lambda}\right)+c_1 n^2
\end{align}
for a large enough $n$. Using (\ref{trivup}) we can now modify $c, c_1$ so that (\ref{end}) holds for all $n$ and $\lambda\le \lambda_0$, which finishes the proof. 
\end{proof}

$ $\\[0pt]
\noindent\textbf{Acknowledgements.} B. Valk\'o was partially supported by the National Science Foundation CAREER award DMS-1053280. The authors  thank B.~Vir\'ag for suggesting them the study of the overcrowding problem. They also thank   P.J.~Forrester, H. Nguyen and J.~Yin for valuable comments.

\def\cprime{$'$}

\end{document}